\documentclass[12pt, reqno]{amsart}
\usepackage{amsmath, amsthm, amscd, amsfonts, amssymb, array, graphicx, color,verbatim}
\usepackage[bookmarksnumbered, colorlinks, plainpages]{hyperref}
\hypersetup{colorlinks=true,linkcolor=red, anchorcolor=green, citecolor=cyan, urlcolor=red,
filecolor=magenta, pdftoolbar=true}
\newtheorem{thm}{Theorem}[section]
 \newtheorem{cor}[thm]{Corollary}
 
 \newtheorem{prop}[thm]{Proposition}
 \theoremstyle{definition}
 
 \theoremstyle{remark}
 \newtheorem{rem}[thm]{Remark}

 \newcommand{\bpf}{\begin{proof}}
\newcommand{\epf}{\end{proof}}

\numberwithin{equation}{section}
\newcommand{\C}{{\mathbb C}}
\newcommand{\B}{{\mathbb B}}
\newcommand{\D}{{\mathbb D}}

\newcommand{\N}{{\mathbb N}}
\newcommand{\Z}{{\mathbb Z}}
\newcommand{\K}{{\kappa}}
\newcommand{\vp}{{\varphi}}

\begin{document}

\setcounter{page}{1}

\title[Applications of Reproducing Kernels in composition operators]{Applications of Reproducing Kernels in composition operators}
\author[P. Kumari, P. Muthukumar and A. Rasila]{Preeti Kumari, P. Muthukumar and Antti Rasila}

\address{Preeti Kumari, Department of Mathematics and Statistics, Indian Institute of Technology,
 Kanpur- 208016, India.}
\email{\textcolor[rgb]{0.00,0.00,0.84}{preetisharma8274@gmail.com, preetik22@iitk.ac.in}}

\address{P. Muthukumar, Department of Mathematics and Statistics, Indian Institute of Technology,
 Kanpur- 208016, India.}
\email{\textcolor[rgb]{0.00,0.00,0.84}{pmuthumaths@gmail.com, muthu@iitk.ac.in}}

\address{Antti Rasila, Department of Mathematics with computer Science, Guangdong Technion-Israel Institute of Technology,  Shantou, Guangdong 515063, P. R. China; and
Department of Mathematics, Technion-Israel Institute of Technology, Haifa
3200003, Israel.}
\email{\textcolor[rgb]{0.00,0.00,0.84}{antti.rasila@iki.fi, antti.rasila@gtiit.edu.cn}}



\date{\today}

\begin{abstract}
In this paper, we illustrate the effectiveness of reproducing kernel Hilbert space techniques in the study of composition operators. For weighted Hardy spaces on the unit disk, we characterize the composition operators whose adjoint is again a composition operator. Using reproducing kernel methods, we obtain a classification of bounded weighted composition operators acting between reproducing kernel Hilbert spaces. We also show that the reproducing kernel techniques yield simpler proofs of several known results, highlighting the role of reproducing kernels as a unifying structural tool in the analysis of composition operators.
\end{abstract}

\maketitle

\section{Introduction}

Let $\mathcal{V}$ be a linear space containing complex valued functions on a nonempty set $X$. For a self map $\vp$ of the set $X$, the \emph{composition operator} $C_\vp$ is defined as $$C_\vp(f)=f \circ \vp, \text{ for all } f\in \mathcal{V}.$$ More generally for a self map $\vp$ of the set $X$ and a complex valued function $\psi$ on the set $X$, the \emph{weighted composition operator} $W_{\vp,\psi}$ is defined as $$W_{\vp,\psi}(f)=\psi\cdot(f \circ \vp), \text{ for all } f\in \mathcal{V}.$$ When the weighting function $\psi$ is identically equal to 1, the operator $W_{\vp,\psi}$ simplifies to a composition operator. When the self map $\vp$ is the identity function on $X$ then the operator reduces to a multiplication operator, denoted as $M_\psi$. In this case $M_\psi(f)=\psi f$.

The study of composition operators acting on a space of holomorphic functions in one or more variables is well explored. Researchers actively explore their basic characteristics, such as boundedness, compactness, and spectral properties, across numerous analytic function spaces, like Hardy, Bergman, and Bloch spaces defined on a variety of domains, for example the unit ball $\B_n$ or the unit polydisc $\D^n$ in $\C^n$, see \cite{cow, shap, zhu ball} and references therein.

In \cite{jaf}, Jafari characterized the boundedness of composition operators on Hardy spaces $H^2(\D^n)$ and weighted Bergman spaces $A^2_\alpha(\D^n)$ using Carleson measures. A similar characterization holds true for the unit ball case, as demonstrated in \cite[Section 3.5]{cow}. The classical Littlewood's theorem (see \cite[page 16]{shap}) states that $C_\vp$ is a bounded linear operator on the Hardy space of the unit disk $H^2(\D)$ for every analytic self-map $\vp$ of $\D$. Moreover,
$$
\|C_\vp\| \leq \sqrt{\frac{1 + |\vp(0)|}{1 - |\vp(0)|}}.
$$

The Hardy and Bergman spaces are fundamental examples of a reproducing kernel Hilbert spaces (RKHS). Previous research has predominantly relied on the operator-theoretic and analytic frameworks of these spaces, while the RKHS framework have received limited attention. The structure provided by its reproducing kernels facilitates computations that would otherwise be more involved. For instance, in \cite{PPJ} an alternative proof of Littlewood’s theorem establishing the boundedness of composition operators on $H^2(\D)$ was given using only reproducing kernel techniques. The inequalities
$$
\|C_{\vp}\|\leq \left\|{\frac{\sqrt{1-|\vp(0)|^2}}{1-\overline{\vp(0)} \vp}}\right\|_\infty \leq \sqrt{\frac{1+|\vp(0)|}{1-|\vp(0)|}}
$$
were also obtained. Moreover, the obtained estimate is sharper than the classical estimate for a broader class of self-maps.

In this paper, we employ reproducing kernel techniques to derive new results concerning composition operators and to reprove some known theorems from a more unified perspective. In Section \ref{sec:inner}, we consider composition operators on the classical Hardy space $H^2(\D)$ induced by inner functions. In this setting the norm of composition operator $C_\vp$ is explicitly known. We show that this norm is not attained by any function in $H^2(\D)$ unless $\vp(0)=0$, in which case the composition operator is an isometry. We also construct a sequence of functions $(f_n)$ for which $\|C_\vp(f_n)\|$ approaches to $\|C_\vp\|$. In Section \ref{sec:adjoint}, we prove that the adjoint of a composition operator $C_\vp$ acting on a weighted hardy space $H^2(\beta)$ is itself a composition operator if and only if $\vp$ is an affine map fixing origin, that is, $\vp(z)=\delta z$ for some $|\delta|\leq1$. In Section \ref{sec:WCO}, we consider weighted composition operators acting between reproducing kernel Hilbert spaces and classify bounded weighted composition operators in terms of reproducing kernels. Finally in Section \ref{sec:several}, we give simpler proof of some known results on composition operators in several complex variables.

\section{Composition operators induced by inner functions}\label{sec:inner}

Recall that the Hardy space $H^2(\D)$ consists of those analytic functions over $\D$ which has square summable coefficients, that is,
$$
H^2(\D)= \left\{f(z)=\sum_{n=0}^\infty a_n z^n, \|f\|:=\left(\sum_{n=0}^\infty |a_n|^2\right)^{1/2}<\infty\right\}.
$$
It is a classical reproducing kernel Hilbert space with kernel function
$$
\K(z,w)=\frac{1}{1-\overline{w}z}, \text{ for }z,w\in(\D).
$$

For $p\in \D$, throughout the article, we use $\alpha_p$ to denote the special disk automorphism
that interchanges $p$ and $0$. That is,
$$
\alpha_p(z):=\frac{p-z}{1-\overline{p}z},\text{ } z\in \D.
$$
Note that for every $z\in\D$, we have the identity
\begin{equation}\label{spe-aut}
1-|\alpha_p(z)|^2=\frac{(1-|p|^2)(1-|z|^2)}{|1-\overline{p}z|^2},
\end{equation}
which we will use frequently.

Let $\vp:\D\rightarrow\D$ be an inner function, i.e.
$|\vp(e^{i\theta})|=1$ for almost all $\theta\in[0,2\pi)$.
In 1968, Nordgren \cite{Nord} computed the operator norm of $C_\vp$ as
\begin{equation}\label{innernorm}
    \|C_\vp\|=\sqrt{\frac{1+|\vp(0)|}{1-|\vp(0)|}}.
\end{equation}

Further if $\vp(0)=0$, then $C_\vp$ is an isometry on $H^2(\D)$ and therefore the
norm  $\|C_\vp\|$ is attained by every function in $H^2(\D)$, other than the zero function.
Does the norm  $\|C_\vp\|$ get attained by some function in $H^2(\D)$ for the case of $\vp(0)\neq 0$?
\begin{thm}\label{inner}
Let  $\vp$ be an inner function with  $\vp(0)\neq 0$ and $f\in H^2(\D)$. Then
$$
\|C_\vp f\|=\sqrt{\frac{1+|\vp(0)|}{1-|\vp(0)|}}\|f\|
$$
if and only if $f\equiv 0$.
That is, $\|C_\vp\|$ is never attained by any function in $H^2(\D)$.
\end{thm}
\begin{proof}
Let $p=\vp(0)$. Take
$\psi=\alpha_p\circ \vp$, so that $\psi$ is an inner function, and $\psi(0)=0$. Consequently, $C_\psi$ is an isometry on
$H^2(\D)$, and

\begin{equation}\label{aut}
\|C_{\alpha_p} f\|=\|C_\psi(C_{\alpha_p} f)\|=\|C_{\vp} f\|.
\end{equation}

Suppose that for some  $f\in H^2(\D)$, one obtains
 $$
\|C_{\vp} f\|=\sqrt{\frac{1+|\vp(0)|}{1-|\vp(0)|}}\|f\|.
$$
By denoting the radial limit function of $f$ again by $f$, we have

$$
\|f\circ \alpha_p\|^2=\int\limits_{0}^{2\pi}|f(\alpha_p(e^{i\theta}))|^2 \frac{d\theta}{2\pi}\cdot
$$
Consider the change of variable $e^{it}:=\alpha_p(e^{i\theta})$.
As $\alpha_p$ is its own inverse, we get $e^{i\theta}=\alpha_p(e^{it})$.
Differentiating this and a little calculation shows that

$$
  d\theta= \frac{\alpha'_p(e^{it})e^{it}}{\alpha_p(e^{it})}dt=\frac{1-|p|^2}{|1-\overline{p}e^{it}|^2} dt.
  $$
  Therefore, for all $f\in H^2(\D)$, we have
  $$
\|f\circ \alpha_p\|^2=\int\limits_{0}^{2\pi}|f(e^{it})|^2 \frac{1-|p|^2}{|1-\overline{p}e^{it}|^2}\frac{dt}{2\pi}\cdot
  $$
In view of the Equation  \eqref{aut}, one has
 $$
\|f\circ \alpha_p\|=\|C_{\alpha_p} f\|=\|C_{\vp} f\|=\sqrt{\frac{1+|p|}{1-|p|}}\|f\|.
$$
Equivalently,
$$
\int\limits_{0}^{2\pi}|f(e^{it})|^2\left(\frac{1+|p|}{1-|p|}-
 \frac{1-|p|^2}{|1-\overline{p}e^{it}|^2}\right)\frac{d\theta}{2\pi}=0.
$$
Since $1-|p|^2\neq 0$, we get
\begin{equation}\label{int}
\int\limits_{0}^{2\pi}|f(e^{it})|^2\left(\frac{1}{(1-|p|)^2}-
 \frac{1}{|1-\overline{p}e^{it}|^2}\right)\frac{d\theta}{2\pi}=0.
\end{equation}
By the triangle inequality, we have $|1-\overline{p}e^{it}|\geq 1-|p|$ and consequently,
the integrand in Equation \eqref{int} is non-negative. Therefore,
$$
|f(e^{it})|^2\left(\frac{1}{(1-|p|)^2}-
 \frac{1}{|1-\overline{p}e^{it}|^2}\right)=0 \mbox{~a.e.~} \partial \D.
 $$
But the equality $|1-\overline{p}e^{it}|= 1-|p|$ is obtained if and only if $t=\arg(p)$, where $\arg$ denotes
 the argument of a complex number. Hence $|f(e^{it})|^2=0$ a.e. on $\partial \D$, and thus
$\|f\|_{H^2(\D)}=0$. That is, $f\equiv 0$, showing the desired claim.
This proves one direction, and the converse is obvious.
\end{proof}

It is well-known that for any inner function $\vp$, we have
$$
\sqrt{\frac{1-|\vp(0)|}{1+|\vp(0)|}}\|f\|\leq \|C_{\vp} f\|,
$$
for all $f\in H^2(\D)$, see \cite[Theorem 3.8]{cow}.

\begin{cor}
 Let  $\vp$ be an inner function with  $\vp(0)\neq 0$ and $f\in H^2(\D)$. Then,
 $$
\sqrt{\frac{1-|\vp(0)|}{1+|\vp(0)|}}\|f\|= \|C_{\vp} f\|
$$
if and only if $f\equiv 0$.
\end{cor}

\begin{proof}
  The proof is similar to that of Theorem \ref{inner},  by making use of  the inequality $|1-\overline{p}e^{it}|\leq 1+|p|$
  instead of $|1-\overline{p}e^{it}|\geq 1-|p|$.
\end{proof}

In literature, when $\vp$ is an inner function, then the norm $\|C_\vp\|$ is computed as follows:
\[\|C_\vp\|=\|C_{\alpha_{\vp(0)}}\|=\|C_{\alpha_{\vp(0)
}}^*\|=\sqrt{\frac{1-|\vp(0)|}{1+|\vp(0)|}}\cdot
\]
Now, can one obtain an explicit sequence $\{f_n\}$ in $H^2(\D)$ such that
$$
\frac{\|C_\vp(f)\|}{\|f\|}\rightarrow\|C_\vp\| \text{ ?}
$$
To answer this, we need the following proposition.

\begin{prop} \label{ineq}
Let $p$ be a nonzero element in $\D$. 
 Then
\begin{itemize}
    \item[(i)] $\displaystyle\frac{\|C_{\alpha_p}\K_{w}\|}{\|\K_{w}\|}\leq \frac{\|C_{\alpha_p}^*\K_{-w}\|}{\|\K_{-w}\|}$
    for all $w=-\frac{pr}{|p|}$, where  $0<r<1;$
    \item[(ii)] $\displaystyle\frac{\|C_{\alpha_p}^*\K_{z}\|}{\|\K_{z}\|} \leq
    \frac{\|C_{\alpha_p} \K_{\alpha_p(z)}\|}{\|\K_{\alpha_p(z)}\|}$ for all $z\in \D.$
\end{itemize}
\end{prop}

\bpf
(i) Let $p=|p|e^{i\zeta}$. Then $w=-e^{i\zeta}r$.
Consider the rotation map $\psi(z)=e^{-i\zeta}z$. Then,
$$
\psi\circ\alpha_p\circ \psi^{-1}(z)=\frac{|p|-z}{1-|p|z},
$$
is the special   automorphism $\alpha_{|p|}$.
As $C_\psi$ is an unitary operator on $H^2(\D)$, we observe that
$$
\|C_{\alpha_{|p|}}(\kappa_{-r})\|=\|C_{\psi^{-1}}C_{\alpha_p} C_\psi(\kappa_{-r})\|=\|C_{\alpha_p}(\kappa_{w})\|
$$
and similarly $\|C_{\alpha_{|p|}}^*(\kappa_{r})\|=\|C_{\alpha_p}^*(\kappa_{-w})\|$.
Hence, without loss of generality, we may assume that $p>0$ and $w=-r$ for $0<r<1$.

To prove the claim, let us first compute $\|C_{\alpha_p}\K_{w}\|$:
$$
C_{\alpha_p}\K_{w}(z)=\frac{1}{1-\overline{w}\alpha_p(z)}=
\frac{1}{1+r(\frac{p-z}{1-pz})}=\frac{1}{1+rp}\cdot\frac{1-pz}{1-\beta z},
$$
where $\beta=\dfrac{p+r}{1+pr}=\alpha_p(-r)$. Define $F(z)=\dfrac{1-pz}{1-\beta z}$. Then
\begin{equation}\label{F}
 \|C_{\alpha_p}(\K_{w})\|^2=\frac{1}{(1+rp)^2}\|F\|^2.
\end{equation}
Now,
$$
\|F\|^2  \displaystyle =\frac{1}{2\pi}\int_0^{2\pi}|F(e^{i\theta})|^2d\theta
=\frac{1}{2\pi}\int_0^{2\pi} \frac{1-2p\cos\theta+p^2}{1-2\beta\cos\theta+\beta^2}d\theta.
$$
Making the change of variable $z=e^{i\theta}$, we see that
$$
\|F\|^2 =\frac{1}{2\pi i}\int_{|z|=1}\frac{p(z-p)(z-\frac{1}{p})}{\beta(z-\beta)(z-\frac{1}{\beta})}\frac{dz}{z}\cdot
$$
As $\beta\in \D$, by using Cauchy's residue theorem,  we obtain

\begin{eqnarray*}
    \|F\|^2  &=&\displaystyle\frac{p}{\beta}\left(1+\frac{(\beta-p)(\beta-\frac{1}{p})}{\beta^2-1}\right)\\
      &=&\displaystyle\frac{p(1+pr)}{p+r}\left(1+\frac{r(1-p^2)}{p(1-r^2)}\right)
      =\displaystyle\frac{1-p^2r^2}{1-r^2}\cdot
\end{eqnarray*}
Substituting the value of $\|F\|$ in Equation \eqref{F}, one gets

$$
\frac{\|C_{\alpha_p}(\K_{w})\|^2}{\|\K_{w}\|^2}
=\frac{1-p^2r^2}{(1+rp)^2} \cdot
$$
Next, with the help of Equation \eqref{spe-aut}, we have
$$
\frac{\|C_{\alpha_p}^*(\K_{-w})\|^2}{\|\K_{-w}\|^2}
=\frac{1-r^2}{1-|\alpha_p(r)|^2}=\frac{(1-pr)^2}{1-p^2}\cdot
$$
Now, for any $p>0$ and $0<r<1$, we see that
$$
\frac{\|C_{\alpha_p}(\K_{w})\|^2}{\|\K_{w}\|^2}\leq
\frac{\|C_{\alpha_p}^*(\K_{-w})\|^2}{\|\K_{-w}\|^2}\cdot
$$
This is because
$$
\frac{1-p^2r^2}{(1+rp)^2} \leq \frac{(1-pr)^2}{1-p^2}
\text{ if and only if } r\leq1.
$$
(ii) The desired claim  follows from the observation that
$$
 \|C_{\alpha_p}^*\K_{z}\|^2
 = \langle C_{\alpha_p}^*\K_{z},C_{\alpha_p}^*\K_{z}\rangle
 = \langle C_{\alpha_p} C_{\alpha_p}^*\K_{z},\K_{z}\rangle
 \leq \|C_{\alpha_p} C_{\alpha_p}^*\K_{z}\|\|\K_{z}\|.
$$
Consequently, we have
$$
\frac{\|C_{\alpha_p}^*\K_{z}\|}{\|\K_{z}\|}
\leq \frac{\|C_{\alpha_p} C_{\alpha_p}^*\K_{z}\|}{\|C_{\alpha_p}^*\K_{z}\|}
= \frac{\|C_{\alpha_p} \K_{\alpha_p(z)}\|}{\|\K_{\alpha_p(z)}\|}\cdot
$$
\epf

\begin{rem}
    Let $\vp$ be an inner function with $\vp(0)=p\neq0$. Let $(r_n)$ be a sequence of real numbers
    such that $0< r_n< 1$ and $r_n\rightarrow 1$ as $n\rightarrow \infty$. Now, define $w_n=-\frac{p}{|p|}r_n$.
    From the Equation \eqref{spe-aut}, we get
     $$
     \frac{1-|w_n|^2}{1-|\alpha_p(w_n)|^2}=\frac{|1-\overline{p}w_n|^2}{1-|p|^2}
     =\frac{(1+|p|r_n)^2}{1-|p|^2}\cdot
     $$
    Now, as $n\rightarrow \infty$,
    $$
    \frac{\|C_{\alpha_p}^*(\K_{w_n})\|^2}{\|\K_{w_n}\|^2}=\frac{1-|w_n|^2}
    {1-|\alpha_p(w_n)|^2}\rightarrow\frac{(1+|p|)^2}{1-|p|^2}=\frac{1+|p|}{1-|p|}
    $$
    and similarly,
    $$
    \frac{\|C_{\alpha_p}^*(\K_{-w_n})\|^2}{\|\K_{-w_n}\|^2}=\frac{1-|w_n|^2}
    {1-|\alpha_p(-w_n)|^2}   \rightarrow\frac{(1-|p|)^2}{1-|p|^2}=\frac{1-|p|}{1+|p|}\cdot
    $$
Hence, as $n\rightarrow\infty$, by the Sandwich theorem and Theorem \ref{ineq}, we get
$$
\frac{\|C_{\alpha_p} \K_{\alpha_p(w_n)}\|}{\|\K_{\alpha_p(w_n)}\|}\rightarrow \sqrt{\frac{1+|p|}{1-|p|}}=\|C_{\alpha_p}\|,
$$
as well as
$$ \frac{\|C_{\alpha_p}(\K_{w_n})\|}{\|\K_{w_n}\|}\rightarrow \sqrt{\frac{1-|p|}{1+|p|}}=\frac{1}{\|C_{\alpha_p}\|}\cdot
$$
Since,
$\|C_\vp(f)\|=\|C_{\alpha_p}(f)\|$ for every $f\in H^2(\D)$,
 we obtained two sequences of normalized kernels such that the norm of their images under $C_\vp$ approach
 to the upper and lower bounds of the operator norm $\|C_\vp\|$, respectively.
\end{rem}

Although the proof is simple, next result is a really interesting generalization of \cite[Theorem 9.4]{cow}.
\begin{prop}\label{norm}
Let $\psi$ be an inner function and let
 $\vp(z)=az\psi(z)+b$ with $b\in \D$ and $|a|+|b|\leq 1$. Consider the composition operator $C_\vp$ on
 $H^2(\D)$. Then,
$$
\|C_{\vp}\|_{H^2(\D)}=\sqrt{\frac{2}{1+|a|^2-|b|^2+\sqrt{(1-|a|^2+|b|^2)^2-4|b|^2}}}\cdot
$$
\end{prop}
\begin{proof}
 Note that $\vp=az\psi+b=(az+b)\circ z\psi$ and consequently,
$$
C_\vp=C_{z\psi}C_{az+b}.
$$
Now, as an immediate consequence of  \cite[Theorem 3.8]{cow}, we see that
$C_{z\psi}$ is an isometry on $H^2(\D)$. Therefore,

$$
\|C_{\vp}\|=\sup\limits_{\|f\|=1}\frac{\|C_{\vp}(f)\|}{\|f\|}=\sup\limits_{\|f\|=1}\frac{\|C_{z\psi}(C_{az+b}(f))\|}{\|f\|}=
\sup\limits_{\|f\|=1}\frac{\|C_{az+b}(f)\|}{\|f\|}=\|C_{az+b}\|.
$$
Now, the desired result follows from \cite[Theorem 9.4]{cow}.
\end{proof}

We end this section with the following remark which follows immediately from Proposition \ref{ineq} (part ii).
\begin{rem}\label{lower bound}
    Let $H(\K)$ and $H(\rho)$ be RKHSs with kernels $\K$ and $\rho$, respectively, on a nonempty set $X$ and $\vp:X\rightarrow X$ be a function.
    If the composition operator $C_{\vp}:H(\K)\rightarrow H(\rho)$ is bounded then
    $$
    \underset{x\in X}{\sup} \frac{\|C_\vp^*\rho_x\|}{\|\rho_x\|}\leq
    \underset{x\in X}{\sup} \frac{\|C_\vp\K_{\vp(x)}\|}{\|\K_{\vp(x)}\|} \leq \|C_\vp\|.$$
\end{rem}

\section{Adjoint of the composition operators}\label{sec:adjoint}

The adjoint of composition operator is rarely a composition operator. Schwarz in \cite[Theorem 2.3]{Sch} showed that it is possible for composition operators defined on $H^2(\D)$ if and only if
$$
\vp(z)=\alpha z, \quad|\alpha|\leq1.
$$
In the next theorem, we give a simpler proof of this result utilizing only the reproducing kernel properties of $H^2(\D).$

\begin{thm}
    Let $\vp$ be a holomorphic self-map of the unit disc $\D$ and $C_\vp:H^2(\D)\rightarrow H^2(\D)$ be
    the corresponding composition operator. Then $C_\vp^*$ is a composition operator if and only if $\vp(z)=\delta z$ for some $|\delta|\leq1$.
\end{thm}
\bpf Assume $C_\vp^*=C_\psi$, for some holomorphic self-map $\psi$ of $\D$. Then for each $z,w \in \D$,

$$\begin{array}{lrcl}
     & C_\vp^*(\K_w)(z) & = & C_\psi(\K_w)(z) \\
    \iff & \displaystyle \frac{1}{1-\overline{\vp(w)}z} &=& \displaystyle \frac{1}{1-\overline{w}\psi(z)}\\
    \iff & \overline{\vp(w)}z &=& \overline{w}\psi(z)
\end{array}$$
Putting $w=0$ gives $\vp(0)=0$ and similarly putting $z=0$ gives $\psi(0)=0$. Hence,
$$\vp=z\vp_1 \text{ and } \psi=z\psi_1$$
for some holomorphic maps $\vp_1$ and $\psi_1$ on $\D$. Now, $$\overline{\vp_1(w)}=\psi_1(z)
\text{ for all }z,w\in \D\backslash\{0\}.$$ This forces $\vp_1\equiv \delta$ for some constant $\delta$.
Hence $\vp(z)=\delta z$ for some $|\delta|\leq1$.

For the other way, let $\vp(z)=\delta z$ for some $|\delta|\leq1$. Let $\psi$ be the holomorphic self-map of $\D$
defined as $\psi(z)=\overline\delta z$. Then, for every $z,w\in \D$ we have
$$
C_\vp^*(\K_w)(z)=\frac{1}{1-\overline{\vp(w)}z}=\frac{1}{1-\overline{\delta w} z}
=\K_w(\overline\delta z)=C_\psi(\K_w)(z).
$$ Hence, $C_\vp^*=C_\psi$.

\epf
Recall, for $\alpha>-1$, the weighted Bergman space $A^2_\alpha$ on $\D$ is defined as the
collection of all holomorphic function on $\D$ for which
$$
\displaystyle \|f\|^2_{A^2_\alpha}:=\int_{\D} |f(z)|^2 (1-|z|^2)^{\alpha} dA(z)<\infty,
$$ where $A$ is the normalized area measure on $\D$.
The spaces $A^2_\alpha$ are reproducing kernel Hilbert spaces with kernels defined by
$$
\K_w(z)=\frac{1}{(1-\overline{w}z)^{\alpha+2}}\cdot
$$
If $f\in A^2_\alpha$ and $f(z)=\sum_{n=0}^\infty a_nz^n$ then
$$
\|f\|^2_{A^2_\alpha}=\sum_{n=0}^\infty \frac{n!\Gamma(2+\alpha)}{\Gamma(n+2+\alpha)}|a_n|^2.
$$
\begin{thm}
    Let $\vp$ be a holomorphic self-map of the unit disc.
    Then the adjoint of the composition operator $C_\vp:A^2_\alpha(\D)\rightarrow A^2_\alpha(\D)$
    is a composition operator itself if and only if $\vp(z)=\delta z$ for some $|\delta|\leq1$.
\end{thm}
\bpf Assume $C_\vp^*=C_\psi$, for some analytic self-map $\psi$ of $\D$. Then for each $z,w \in \D$,
$$\begin{array}{lrcl}
     & C_\vp^*(\K_w)(z) & = & C_\psi(\K_w)(z) \\
    \iff & \displaystyle \frac{1}{(1-\overline{\vp(w)}z)^{\alpha+2}} &=& \displaystyle \frac{1}{(1-\overline{w}\psi(z))^{\alpha+2}}\\
    \iff & (1-\overline{\vp(w)}z)^{\alpha+2} &=& (1-\overline{w}\psi(z))^{\alpha+2}
\end{array}$$

Note that $|1-\overline{\vp(w)}z)|\geq 1-|\overline{\vp(w)}z|>0$ and similarly $|1-\overline{w}\psi(z)|>0$.
Hence, principal branch of the complex logarithm (Log) exits and
$$
\displaystyle e^{(\alpha+2)\mathrm{Log}(1-\overline{\vp(w)}z)}=e^{(\alpha+2)\mathrm{Log}(1-\overline{w}\psi(z))}\text{ for all }z,w\in \D.
$$
Thus, for each pair $(z,w)\in \D\times\D$ there exists a $n(z,w)\in \Z$ such that
$$
(\alpha+2)\mathrm{Log}(1-\overline{\vp(w)}z)-(\alpha+2)\mathrm{Log}(1-\overline{w}\psi(z))= 2\pi i n(z,w).
$$
Fix $w\in \D$. Then the holomorphic map
$$
(\alpha+2)\mathrm{Log}(1-\overline{\vp(w)}z)-(\alpha+2)\mathrm{Log}(1-\overline{w}\psi(z))
$$
takes values in $\{2\pi i n:n\in\Z\}$. As a consequence of the open mapping theorem we get for each fixed $w\in \D$
$$n(z_1,w)=n(z_2,w)\text{ for all }z_1,z_2\in\D.$$
Further, let $\displaystyle\vp(z)=\sum_{n=0}^\infty a_n z^n$, and define $\displaystyle\tilde\vp(z)=\sum_{n=0}^\infty \overline{a_n} z^n$.
Then $\tilde\vp$ is holomorphic on $\D$ and $\overline{\vp(w)}=\tilde\vp(\overline{w})$. Now, fix $z\in \D$. Then

$$
(\alpha+2)\mathrm{Log}(1-\tilde\vp(\overline{w})z)-(\alpha+2)\mathrm{Log}(1-\overline{w}\psi(z))
$$
is a holomorphic map of $\overline{w}$, taking values in $\{2\pi i n:n\in\Z\}$.
Again by the open mapping theorem we get for each fixed $z\in \D$,
$$
n(z,w_1)=n(z,w_2)\text{ for all }w_1,w_2\in\D.
$$
Note that $n(0,0)=0$. Finally, for all $z,w\in \D$ we get
$\mathrm{Log}(1-\overline{\vp(w)}z)=\mathrm{Log}(1-\overline{w}\psi(z))$
that is same as $1-\overline{\vp(w)}z =1-\overline{w}\psi(z)$.
The result follows from the calculations made in the previous theorem.
\epf

Note that both the spaces considered above, namely the Hardy space $H^2(\D)$ and the weighted Bergman space $A^2_\alpha(\D)$ are weighted Hardy spaces, which is defined as follows: Let $(\beta)_{n>0}$ be a sequence with $\beta_0=1$, $\beta_n>0$ for all $n$, and
$\liminf \beta_n^{1/n}\geq1$, then the weighted Hardy space $H^2(\beta)$ is defined as
$$
H^2(\beta)=\left\{ f(z)=\sum_{n=1}^\infty a_n z^n: \sum_{n=1}^\infty |a_n|^2\beta_n^2 \right\}
$$
with $\langle f,g\rangle=\sum_n a_n\overline{b_n}\beta_j^2$.
It is well know that $H^2(\beta)$ is a reproducing kernel Hilbert space,
consisting of functions analytic on $\D$, with kernels
$$
\K_w(z)=\sum_{n=0}^\infty \frac{\overline{w}^n}{\beta_n^2}z^n
$$
 The next theorem shows that the result of the last two theorem is indeed true in general for all weighted Hardy spaces.

\begin{thm}
    Let $H^2(\beta)$ be a weighted Hardy space on $\D$, with weights $\beta=(\beta_j)$.
    Let $\vp$ be a holomorphic self-map of the unit disc $\D$ such that the composition operator $C_\vp:H^2(\beta)\rightarrow H^2(\beta)$ is bounded.
    Then $C_\vp^*$ is a composition operator itself if and only if $\vp(z)=\delta z$ for some $|\delta|\leq1$.
\end{thm}
\bpf
Assume $C_\vp^*=C_\psi$, for some self-map $\psi$ of $\D$. Then for each $z,w \in \D$,
\begin{equation}\label{adj-eqn}
    C_\vp^*(\K_w)(z) = C_\psi(\K_w)(z).
\end{equation}
In particular, when $w=0$ we get $\K_{\vp(0)}=1$, hence, $\vp(0)=0$.
Similarly, when $z=0$, we get $\psi(0)=0$. Thus, for some holomorphic maps $\vp_1$ and $\psi_1$ on $\D$,
$$\vp=z\vp_1 \text{ and } \psi=z\psi_1.$$


Define $f(z,w)=\K_{\vp(\overline{w})}(z)-\K_{\overline{w}}(\psi(z))$. Then $f$ is an analytic function on $\D^2$ using the fact that $\K_{\overline w}(z)=\K_{\overline z}(w)$ for every $(z,w)\in\D^2$. Now, putting the values of $\K,\vp$ and $\psi$ we obtain,
$$
f(z,w)=zw\left(\sum_{n=1}^\infty \frac{(wz)^{n-1}\overline{\vp_1(\overline{w})}^n}{\beta_n}-
\sum_{n=1}^\infty \frac{(wz)^{n-1}\psi_1(z)^n}{\beta_n}\right)\cdot
$$
By \eqref{adj-eqn}, $f\equiv0$. Hence,
$$
g(z,w):=\sum_{n=1}^\infty \frac{(wz)^{n-1}\overline{\vp_1(\overline{w})}^n}{\beta_n}-
\sum_{n=1}^\infty \frac{(wz)^{n-1}\psi_1(z)^n}{\beta_n}=0\text{ for all }z\neq0,w\neq0.
$$
Since, $g$ is also analytic on $\D^2$, by Identity theorem for several variables (see \cite[Conclusion 1.2.12]{vol}), we conclude that $g\equiv0$, i.e.
$$
\sum_{n=1}^\infty \frac{(wz)^{n-1}\overline{\vp_1(\overline{w})}^n}{\beta_n}=
\sum_{n=1}^\infty \frac{(wz)^{n-1}\psi_1(z)^n}{\beta_n}\text{ for all }(z,w)\in\D^2.
$$
In particular,
$$
\sum_{n=1}^\infty \frac{(\overline{w}z)^{n-1}\overline{\vp_1(w)}^n}{\beta_n}=
\sum_{n=1}^\infty \frac{(\overline{w}z)^{n-1}\psi_1(z)^n}{\beta_n}\text{ on }\D^2.
$$
Putting $z=0$ gives that
$$
\frac{\overline{\vp_1(w)}}{\beta_1}=
\frac{\psi(0)}{\beta_1}\text{ for all }w\in\D.
$$
Therefore, $\vp_1\equiv \delta$ for some constant $\delta$. Hence $\vp(z)=\delta z$ for some $|\delta|\leq1$.

For converse, let $\psi(z)=\overline{\delta}z$. We observe that for every $w\in \D$,
$$C_\vp^*(\K_w)=C_\psi(\K_w).$$
Hence, $C_\vp^*=C_\psi$ on $H^2(\beta)$.
\epf

\section{Weighted composition operators}\label{sec:WCO}

In \cite{jury}, Jury reproved the boundedness of composition operators acting on $H^2(\D)$ using reproducing kernel techniques. Following this, in \cite{chu}, Chu gave some sufficient condition in terms of reproducing kernels for the boundedness of composition operators acting on the Hardy space over bisdisc $H^2(\D^2)$ (see section \ref{sec:several} for definition). Unlike composition and multiplication operators, weighted composition operators received limited attetion in this regard. Although the boundedness of weighted composition operator acting between several function spaces has been characterized in terms of carleson measure condition, see \cite{manu}, \cite{zhao}, \cite{kellay}, \cite{stevic} and the references therein, Carleson measures are often difficult to compute explicitly; even for well-behaved functions such as inner functions. It is therefore desirable to have alternative characterizations. In this section, we will present a characterization of bounded weighted composition operators acting between RKHSs by utilizing their respective reproducing kernels. We will begin with showing that the image of an RKHS under a weighted composition operator is itself an RKHS, with a suitably modified kernel.

Prior to presenting our main theorems, we will review some fundamental concepts of reproducing kernels. For a more comprehensive treatment of reproducing kernel Hilbert spaces, we refer the reader to \cite{paul}.
\begin{prop}\label{property}
    Let $X$ and $S$ be nonempty sets. If $\K_1$ and $\K_2$ are kernel functions on $X$, then
    \begin{itemize}
        \item[(i)] The sum $\K_1+\K_2$ and the pointwise product $\K_1\cdot\K_2$ is a kernel function on $X$.
        \item[(ii)] If $\vp:S\rightarrow X$ is a function, then $\K_1(\vp(\cdot),\vp(\cdot))$ is a kernel function on S.
        \item[(iii)] For any complex valued map $f$ on $X$, the  map $(x,y)\mapsto f(x)\overline{f(y)}$ is a kernel function on $X$.
    \end{itemize}
\end{prop}

\begin{thm}\cite[Theorem 3.11]{paul}\label{in}
    Let $H(\K)$ be an RKHS, on a nonempty set $X$, with reproducing kernel $\K$ and let $f:X\rightarrow \C$ be a function. Then $f\in H(\K)$ with $\|f\| \leq c$ if and only if $c^2\K(x,y)-f(x)\overline{f(y)}$ is a kernel function on $X$.
\end{thm}

Composition operators on RKHS can be characterized by considering the set $\{\K_x: x\in X\}$ as given in following theorem.

\begin{thm} \cite[Theorem 1.4]{cow}\label{adjoint}
    Let $T$ be a bounded linear operator mapping an RKHS into itself, then $T$ is a composition operator
    if and only if the set $\{\K_x: x\in X\}$ is invariant under $T^*$. Moreover, $T^*(\K_x)=\K_{\vp(x)}$, when $T= C_{\vp}$.
\end{thm}
\begin{thm}\label{wco1}
Let  $\K$ be a kernel  on a set $X$,
$\psi$ be a complex valued map on a set $S$ and
$\vp: S\rightarrow X$  be a map. Then $\rho(x,y):=\psi(x)\overline{\psi(y)}\K(\vp(x),\vp(y))$ is a kernel function on $S$ and the corresponding reproducing kernel Hilbert space $H(\rho)$ is the image of $H(\K)$ under the weighted composition operator $W_{\vp,\psi}$. Additionally, for every $g\in H(\rho)$, we have $\|g\|_{H(\rho)}=\min\{\|f\|_{H(\K)}: g=W_{\vp,\psi}(f)\}$.
\end{thm}
\bpf
It is clear from Proposition \ref{property} that $\rho$ is a kernel function on $S$. Fix $f\in H(\K)$. Then by Theorem \ref{in}, $$\|f \|_{H(\K)}^2\K(x,y)-f(x)\overline{f(y)}\geq 0.$$ 
Thus by part (iii) of Proposition \ref{property},
$$
\|f\|_{H(\K)}^2\K(\vp(s),\vp(t))-f(\vp(s))\overline{f(\vp(t))}\geq 0.
$$
 Now, after multiplying it with the another kernel function $\psi(s)\overline{\psi(t)}$ we get that
 $$
 \|f \|_{H(\K)}^2\psi(s)\overline{\psi(t)}\K(\vp(s),\vp(t))-
 (\psi \cdot(f\circ\vp))(s)\overline{(\psi \cdot(f\circ\vp))(t)}\geq 0.
 $$ Hence for all $f\in H(\K)$,
 $$
 \psi \cdot(f\circ\vp) \in H(\rho)   \text{~with~}  \|\psi \cdot(f\circ\vp)\|_{H(\rho)} \leq \|f \|_{H(\K)}.
  $$
  It follows that $W_{\vp,\psi}(H(\K))\subseteq H(\rho)$.

 For the other way inclusion, consider the linear map $T: V \subset H(\rho)\rightarrow H(\K)$ defined as
 $$
 T(\rho_t)=\overline{\psi(t)}\K_{\vp(t)},
 $$
 where $V=\mathrm{span}\{\rho_t:t\in S\}$.
For any arbitrary function $\sum_{i=1}^{n}\alpha_i\rho_{t_i}\in V$, we have

$$\begin{array}{ccl}
     \displaystyle\left\|T\left(\sum\limits_{i=1}^{n}\alpha_i\rho_{t_i}\right)\right\|_{H(\K)}^2 &=&  \displaystyle\left\langle\sum\limits_{i=1}^{n}\alpha_iT(\rho_{t_i}),\sum_{j=1}^{n}\alpha_jT(\rho_{t_j})\right\rangle\\
    &=&  \displaystyle\sum_{i,j=1}^{n} \alpha_i \overline{\alpha_j} \left\langle \overline{\psi(t_i)}\K_{\vp(t_i)},\overline{\psi(t_j)}\K_{\vp(t_j)}\right\rangle\\
     &=& \displaystyle\sum_{i,j=1}^{n} \alpha_i \overline{\alpha_j} \psi(t_j) \overline{\psi(t_i)} \K(\vp(t_j),\vp(t_i))\\
     &=& \displaystyle\sum_{i,j=1}^{n} \alpha_i \overline{\alpha_j} \rho(t_j,t_i)= \left\| \sum_{i=1}^{n}\alpha_i\rho_{t_i}\right\|_{H(\rho)}^2.
\end{array}$$
Therefore, $T$  is an isometry on $V$, which is a dense subspace of $H(\rho)$.
Consequently, $T$ extends to a linear isometry on $H(\rho)$. We denote this extension by $T$ itself.
Now, $W_{\vp,\psi}\circ T:H(\rho)\rightarrow H(\rho)$ is linear and, for every $t\in S$,
$$
(W_{\vp,\psi}\circ T)(\rho_t)=W_{\vp,\psi}(\overline{\psi(t)}\K_{\vp(t)})=
\overline{\psi(t)} \psi\cdot(\K_{\vp(t)}\circ \vp)= \rho_t.
$$ Thus, $W_{\vp,\psi}\circ T$ is the identity map on $V$, hence, on $H(\rho)$.
Finally, for every $g \in H(\rho)$,
$$
g=(W_{\vp,\psi}\circ T)(g)=W_{\vp,\psi}(T(g))\in W_{\vp,\psi}(H(\K))
\text{ with } \|g\|_{H(\rho)}=\|T(g)\|_{H(\K)}.
$$ This yields that
$$
H(\rho)= W_{\vp,\psi}(H(\K))\text{ and }\|g\|_{H(\rho)}=\min\{\|f\|_{H(\K)}: g=W_{\vp,\psi}(f)\}.
$$
\epf

The above theorem being established, the boundedness of the operator $W_{\vp,\psi}$ can be obtained by
showing that its range is contained in the target space, due to the fact that weighted composition operators
are inherently closed operators. The subsequent theorem provides condition for the inclusion of one RKHS within another.
\begin{thm}\label{subset}\cite[Theorem 5.1]{paul}
    Let $X$ be a nonempty set and $\K_1, \K_2$ be kernels on $X$. Then $H(\K_1)\subseteq H(\K_2)$ with
    $\|f\|_{H(\K_2)}\leq c\|f\|_{H(\K_1)}$ for all $f\in H(\K_1)$, if and only if $c^2\K_2-\K_1$ is a kernel function  on $X$.
\end{thm}
We can now present, as desired, the characterization of bounded weighted composition operators through the properties of
their associated kernels by making use of the results of the two preceding theorems.
\begin{thm}\label{wco2}
    Let $X_i$, $i=1,2$, be nonempty sets, $\vp:X_2\rightarrow X_1$ and $\psi:X_2\rightarrow \C$ be functions and $\K_i$ be a
    kernel on $X_i$, $i=1,2$. Then $W_{\vp,\psi}:H(\K_1)\rightarrow H(\K_2)$, defined as $W_{\vp,\psi}(f)=\psi \cdot (f\circ\vp)$,
    is a bounded operator with $\|W_{\vp,\psi}\|\leq c$ if and only if \begin{equation}\label{a1}
    c^2\K_2(x,y)-\psi(x)\overline{\psi(y)}\K_1(\vp(x),\vp(y))
    \end{equation} is a kernel function on $X_2$.
\end{thm}

\bpf
Define $\displaystyle\K_3(x,y)=\psi(x)\overline{\psi(y)}f(\vp(x))\overline{f(\vp(y))}$.
Then by Theorem \ref{wco1},
$$
H(\K_3)=W_{\vp,\psi}(H(\K_1))
$$
and for every $g \in H(\K_3)$, there exists a $f\in H(\K_1)$ such that
$g=\psi \cdot (f\circ\vp)$ with $\|f\|_{H(\K_1)}=\|g\|_{H(\K_3)}$.
Now, if $W_{\vp,\psi}$ is bounded with  $\|W_{\vp,\psi}\|\leq c$, then
$W_{\vp,\psi}(H(\K_1))=H(\K_3)\subseteq H(\K_2)$ as well as
$$
\|g\|_{H(\K_2)}\leq c\|f\|_{H(\K_1)}= c\|g\|_{H(\K_3)} \text{ for all } g\in H(\K_3).
$$
Therefore, Theorem \ref{subset} implies
$$
c^2\K_2(x,y)-\psi(x)\overline{\psi(y)}\K_1(\vp(x),\vp(y))\geq0.
$$

For the other direction,
assume that the map given in \ref{a1} is a kernel function   for some constant $c>0$.
For $f\in H(\K_1)$, by Theorem \ref{in}, it follows that
$$
\|f\|^2_{H(\K_1)}\K_1(x,y)-f(x)\overline{f(y)}\geq 0.
$$  Then Proposition \ref{property} implies that
\begin{equation}\label{a2}
    \|f\|^2_{H(\K_1)}\psi(x)\overline{\psi(y)}\K_1(\vp(x),\vp(y))-\psi(x)\overline{\psi(y)}f(\vp(x))\overline{f(\vp(y))}\geq 0.
\end{equation}
Multiplying \ref{a1} with $\|f\|^2$ and adding it to \ref{a2} gives that
$$
\|f\|^2_{H(\K_1)}c^2\K_2(x,y)-\psi(x)\overline{\psi(y)}f(\vp(x))\overline{f(\vp(y))}\geq 0.
$$
Thus, $\psi \cdot(f\circ\vp)\in H(\K_2)$ and
$$
\|\psi \cdot(f\circ\vp)\|_{H(\K_2)} \leq c\|f\|_{H(\K_1)} \mbox{~ for every~ } f\in H(\K_1).
$$
Hence $W_{\vp,\psi}:{H(\K_1)}\rightarrow{H(\K_2)}$ is a bounded operator with $\|W_{\vp,\psi}\|\leq c$.
\epf

Although, showing that the two variable function in \eqref{a1} is a kernel function, for given $\vp$ and $\psi$,
is not straightforward in general. However, when $W_{\vp,\psi}$ is bounded, it is indeed a kernel function, providing us with a valuable source of examples. For instance, if $\psi$ maps $\D^n$ into $\D$ then $$(1-\overline{\psi(w)}\psi(z))\prod_{i=1}^n\frac{1}{1-\overline{w_i}z_i}$$ is a kernel function on $\D^n$.

\section{Composition operators in several variables}\label{sec:several}

Let $\D^n= \D \times \D \times \cdots \times \D$ (n times) be the unit polydisc in $\C^n$ and $H(\D^n)$ be the collection of all holomorphic functions on $\D^n$. The Hardy space $H^2(\D^n)$ is defined as the collection of all $f\in H(\D^n)$, for which
$$
\|f\|^2_{H^2(\D^n)}:=\underset{0\leq r<1}{\sup}
\int_{\mathbb{T}^n}|f(rz)|^2 dm_n(z)<\infty,
$$
where, $\mathbb{T}^n$ is the distinguished boundary of $\D^n$
and $m_n$ denotes the normalized Lebesgue area measure on $\mathbb{T}^n$.
The space $H^2(\D^n)$ is a reproducing kernel Hilbert space, see \cite[Page 4]{rud poly}, with the kernel function
$$
 \K(z,w)= \prod\limits_{i=1}^n \frac{1}{1-\overline{w_i}z_i}\cdot
$$
For an $\alpha > -1$, the weighted Bergman space $A^2_{\alpha}(\D^n)$ is defined as the collection of all $f\in H(\D^n)$, for which
$$
 \|f\|^2_{A^2_\alpha(\D^n)}:=\int_{\D^n}
|f(z)|^2\prod\limits_{i=1}^n (1-|z_i|^2)^{\alpha} dA_n(z)<\infty,
$$
where $A_n$ is the normalized Lebesgue volume measure on $\D^n$. For each $\alpha >-1$, $A^2_\alpha(\D^n)$ is an reproducing kernel Hilbert space, see \cite[Page 821]{stessin}, with kernel function
$$
\K(z,w)=\prod\limits_{i=1}^n\dfrac{1}{(1-\overline{w_i} z_i)^{\alpha+2}}\cdot
$$
The Boundedness of composition operators acting on spaces described above has been well studied. Here we present simpler proofs of a few results using reproducing kernel structure.

Lower bound for bounded composition operators on $H^2(\D^n)$ has been given in \cite[Proposition 7]{jaf}. This can be obtained easily using Proposition \ref{lower bound}  and similar bound can be obtained for $A^2_\alpha(\D^n)$.

\begin{cor}\label{lower bound 1} Let $\vp=(\vp_1,\vp_2,\ldots,\vp_n): \D^n \rightarrow \D^n$ be holomorphic.
    \begin{itemize}
        \item  If $C_\vp$ is bounded on $H^2(\D^n)$ then $$\displaystyle\sup_{z\in\D^n}\left\{\prod_{i=1}^n
        \frac{1-|z_i|^2}{1-|\vp_i(z)|^2}\right\}\leq \|C_\vp\|^2.$$
        \item If $C_\vp$ is bounded on $A^2_\alpha(\D^n)$ then $$\displaystyle\sup_{z\in\D^n}\left\{
        \prod_{i=1}^n\left(\frac{1-|z_i|^2}{1-|\vp_i(z)|^2}\right)^{\alpha+2}\right\}\leq \|C_\vp\|^2.$$
    \end{itemize}
\end{cor}

For $f(z)=\sum\limits_{s \in \N_0^n} \hat{f}(s) z^s\in A^2_{\alpha}(\D^n)$ define
$$\|f\|_*^2=\sum\limits_{s \in \N_0^n} |\hat f (s)|^2 \prod\limits_{i=1}^n(s_i+1)^{-1-\alpha}.$$
Then $\|\cdot\|_*$ defines an equivalent norm on $A^2_{\alpha}(\D^n)$ and $(A^2_{\alpha}(\D^n),\|\cdot\|_*)$
is an RKHS with kernel function
$$
\K(z,w)=\sum\limits_{|s|=0}^{\infty} (\overline{w}z)^s\prod\limits_{i=1}^n(s_i+1)^{1+\alpha}.
$$
In \cite[Proposition 11]{jaf} Jafari stated that if $C_\vp$ is bounded on $(A^2_{\alpha}(\D^n),\|\cdot\|_*)$ then $\|C_\vp\|^2 \geq A$, where
$$
A=\sup_{z\in\D^n}\frac{\sum\limits_{|s|=0}^{\infty} |\vp(z)|^{2s}\prod\limits_{i=1}^n
(s_i+1)^{-1-\alpha}}{\sum\limits_{|s|=0}^{\infty} |z|^{2s}\prod\limits_{i=1}^n(s_i+1)^{1+\alpha}}\cdot
$$
Now we give a better lower bound in this case.

\begin{prop}
    If $C_\vp$ is bounded on $(A^2_{\alpha}(\D^n),\|\cdot\|_*)$ then
    $$
    A\leq \sup_{z\in\D^n}\left\{\frac{\sum\limits_{|s|=0}^{\infty}
    |\vp(z)|^{2s}\prod\limits_{i=1}^n (s_i+1)^{1+\alpha}}{\sum\limits_{|s|=0}^{\infty}
    |z|^{2s}\prod\limits_{i=1}^n(s_i+1)^{1+\alpha}}\right\}\leq \|C_\vp\|^2.
    $$
\end{prop}
\bpf Given $\alpha>0$ and $s_i\geq0$, it follows that
 $$(s_i+1)^{-1-\alpha}\leq1\leq(s_i+1)^{1+\alpha}.$$
Consequently, the first inequality holds. The second inequality follows from Proposition \ref{lower bound}.
\epf

Figura, in \cite[Lemma $5^\prime$]{fig}, gave a lower bound for the Hardy space of the ball,
which can be improved directly from Proposition \ref{ineq}.
\begin{prop}
    If $C_\vp$ is bounded on $H^2(\B_n)$ then $\|C_\vp\|\geq (1-|\vp(0)|^2)^{\frac{-n}{2}}$.
\end{prop}
\bpf
$\displaystyle\|C_\vp\|\geq \frac{\|C_\vp^*\rho_0\|}{\|\rho_0\|}=
\frac{\|\rho_{\vp(0)}\|}{\|\rho_0\|}=(1-|\vp(0)|^2)^{\frac{-n}{2}}
\geq \{2(1-|\vp(0)|)\}^{\frac{-n}{2}}$.
\epf

The results presented in this paper highlight the effectiveness of reproducing kernel Hilbert space techniques in the analysis of composition and weighted composition operators. These methods indicates promising directions for further study, especially on other function spaces and in multivariable settings.

\vspace{0.2in}
\noindent\textbf{Acknowledgement:}
The first author acknowledges the support of the Prime Minister's
Research Fellowship scheme (PMRF ID. 2303407).
The third author was partially supported by the Li Ka Shing Foundation STU-GTIIT Joint Research Grant (Grant no. 2024LKSFG06) and the NSF of Guangdong Province (Grant no. 2024A1515010467).

\end{document}